\documentclass[leqno,a4paper,11pt]{amsart}

\pdfoutput=1

\usepackage{pinlabel}
\usepackage{cite}
\usepackage{latexsym,amsfonts,amssymb,verbatim,mathrsfs,amsthm}
\usepackage{amsmath,amsthm,amssymb,latexsym,graphics,textcomp, enumerate}
\usepackage{eucal,eufrak}
\usepackage{colonequals}
\usepackage{graphicx}
\usepackage{url}
\usepackage{fullpage}
\usepackage{multicol}
\theoremstyle{plain}
\newtheorem{theorem}{Theorem}[section]
\newtheorem{lemma}[theorem]{Lemma}
\newtheorem{corollary}[theorem]{Corollary}
\newtheorem*{remark*}{Remark}

\newtheorem{proposition}[theorem]{Proposition}

\theoremstyle{definition}

\def\1{\mathbf 1}

\def\EL{\mathcal{EL}}

\def\A{\mathcal A}

\def\CS{\mathcal {C}(S)}
\def\ACS{\mathcal {AC}(S)}
\def\AS{\mathcal {A}(S)}

\def\ELS{\mathcal {EL}(S)}
\def\EL0S{\mathcal {EL}_{0}(S)}
\def\DACS{\partial \mathcal {AC}(S)}
\def\DAS{\partial \mathcal {A}(S)}

\def\GS{\mathcal{G}(S)}

\title{INFINITE UNICORN PATHS AND GROMOV BOUNDARIES }

\author{WITSARUT PHO-ON}

\address{\tt  Department of Mathematics, University of Illinois at Urbana-Champaign,
	1409 West Green Street, Urbana, IL 61801
	\newline \indent http://www.math.uiuc.edu/\~{}phoon2/} \email{\tt phoon2@illinois.edu}

\thanks{The author was partially supported by the NSF grant DMS-1207183.}  

\begin{document}
\maketitle
\nocite{*}

\begin{abstract}
	We extend the notion of unicorn paths between two arcs introduced by Hensel, Przytycki and Webb to the case where we replace one arc with a geodesic asymptotic to a lamination. Using these paths, we give new proofs of the results of Klarreich and Schleimer identifying the Gromov boundaries of the curve graph and the arc graph, respectively, as spaces of laminations.
\end{abstract}
\section{INTRODUCTION}

The goal of this paper is to provide direct elementary proofs of results of Klarreich and Schleimer identifying the Gromov boundaries of the arc and curve graph $\ACS$ and the arc graph $\AS$, respectively. Our proofs use the tools developed by Hensel, Przytycki and Webb in their elementary proofs of hyperbolicity of both $\ACS$ and $\AS$\cite{A}. We begin by recalling Klarreich's Theorem \cite{F}; see also \cite{Ham} and \cite{Alex}.

\vspace{0.2 cm}
 \begin{theorem}[Klarreich] \label{main1}	There is a $Mod(S)$--equivariant homeomorphism $F \colon \ELS \to \DACS$.  Furthermore, if $\{a_n\} \in \ACS$ is a sequence converging to $F(L)$, then any Hausdorff accumulation point of $\{a_n\}$ in $\mathcal{G}(S)$ contains $L$. 
\end{theorem}
Here $Mod(S)$ is the mapping class group of $S$, $\GS$ is the set of all geodesic laminations, and $\ELS$ is the set of all ending laminations. Recently, Schleimer proved the analogous result for $\AS$, see \cite{Saul}. To state this, we must consider a larger space of laminations $\EL0S \supseteq \ELS$. The topology on $\EL0S$ and $\ELS$ is the Thurston topology \cite{Canary}, also called the coarse Hausdorff topology in \cite{Ham}.

 \begin{theorem}[Schleimer] \label{main}	There is a $Mod(S)$--equivariant homeomorphism $F \colon \EL0S \to \DAS$.  Furthermore, if $\{a_n\} \in \AS$ is a sequence converging to $F(L_0)$, then any Hausdorff accumulation point of $\{a_n\}$ in $\mathcal{G}(S)$ contains $L_0$.
 	
 \end{theorem}
 The outline of this paper is as follows. In section 2, we recall some basic definitions and results about Gromov boundaries, laminations, arc and curve graphs, and arc graphs. Some definitions and results about unicorn paths are also included in this section. In section 3, we define infinite unicorn paths and the proof of Theorem \ref{main} is given. In Section 4, we provide the slight modification of the proof of Theorem \ref{main} necessary for Theorem \ref{main1}. 
 
 \vspace{0.3 cm}
 \textbf{Acknowledgments.}  I would like to thank my advisor Christopher J. Leininger for guidance, support and encouragement. I would also like to thank the referee for his/her suggestions.

\section{PRELIMINARIES}

\subsection{Gromov boundaries}Let $X$ be a $\delta$--hyperbolic geodesic metric space. Fix a base point $o$ in $X$. For $x,y \in X$, define the Gromov product \[(x\cdot y)_o =\frac{1}{2}(d(x,o)+d(y,o)-d(x,y)).\]
If $[x,y]$ is a geodesic from $x$ to $y$, then $|d(o,[x,y])-(x,y)_o|\leq 2\delta$. Given two sequences of points in $X$, $\{x_n\}$ and $\{y_n\}$, they are said to be equivalent if $\displaystyle{\liminf_{i,j\rightarrow \infty}(x_i\cdot y_j)_o = \infty}$. Denote $[\{x_n\}]$ the equivalence class of $\{x_n\}$. Define the {\em Gromov boundary} of $X$ by
\[\partial X = \{[\{x_n\}] | \liminf_{i,j \rightarrow \infty} (x_i \cdot x_j)_o =\infty \}.\]
There is a metric on $\partial X$ such that distinct points $[\{x_n\}]$ and $[\{y_n\}]$ in $\partial X$ are close if and only if $\liminf_{i,j \rightarrow \infty}  (x_i \cdot x_j)_o$ is large. See \cite{I} for more details.
\subsection{Arc and curve graph and arc graph} \label{arc}
Throughout, we let $S$ be an oriented connected hyperbolic surface of finite area with finitely many punctures. We consider proper arcs and closed curves on $S$ that are simple and essential. The {\em arc and curve graph} $\ACS$ is the graph whose vertices are isotopy classes of propers arcs and curves on $S$. Two vertices are connected by an edge in $\ACS$ if they are realized disjointly. There are two subgraphs of $\ACS$ we will consider.  The {\em curve graph} $\CS$ is the largest subgraph whose vertex set is the set of isotopy classes of curves, and the {\em arc graph} $\AS$ is the largest subgraph whose vertex set is the set of isotopy classes of arcs.  The inclusion of $\CS$ into $\ACS$ is a quasi--isometry while $\AS$ into $\ACS$ is not. See \cite{D} and \cite{Masur} for more details.

 We say that two arcs or curves are in minimal position if they intersect minimally in their isotopy classes. We always realize isotopy classes of arcs and curves by their complete geodesic representatives, which are in minimal position. Let $S^0$ be a compact subsurface of $S$ obtaining by removing small open horoball cusp neighborhoods around each puncture so that any simple complete geodesic in $S$ is contained in $S^0$ or intersects $S\smallsetminus S^0$ in rays. Whenever we parametrize a bi--infinite geodesic $l$ with one end at a puncture, we require this to have unit speed, and to have $l(-\infty,0)$ being a ray in $S\smallsetminus S^0$ with $l(0)\in \partial S^0$.

\subsection{Laminations} \label{lamination} A {\em geodesic lamination on $S$} is defined to be a closed subset of $S$ which is a disjoint union of simple complete geodesics. Let $L$ be a geodesic lamination. We say {\em$L$ fills a subsurface $Y$ of $S$} if $L\subseteq Y$ and every simple closed geodesic on $Y$ intersects $L$ transversely, and $L$ is called {\em minimal} if every leaf of $L$ is dense in $L$. Any minimal lamination is connected. For a parametrized simple geodesic $l$ starting at a puncture (see Section \ref{arc} for our convention on parametrization), $l$ is said to be {\em asymptotic to $L$} if $l\pitchfork L =\varnothing$ and $\displaystyle{\lim_{t\rightarrow \infty}} d(l(t), L)=0$. We let $L'\subseteq L$ be $L$ with all isolated leaves removed, and call it the {\em derived lamination} of $L$. For more on geodesics laminations, see \cite{Canary} and \cite{C}.

To state the following proposition, we first define {\em a crown} and {\em a punctured crown} to be complete hyperbolic surfaces with finite area and geodesic boundary, which are homeomorphic to $(S^1\times [0,1])\setminus A$ and $(S^1\times (0,1])\setminus A$, respectively, where $A$ is a finite subset of $S^1 \times \{1\}$; see Figure \ref{figure0}. Let $L$ be a minimal lamination which is not a simple closed geodesic and $P$ be a maximal collection of disjoint simple closed geodesics such that $P\cap L=\varnothing$. We can see that each component of $S\setminus (P\cup L)$ is the interior of a complete surface of finite area with geodesic boundary. By Theorem 2.10 of \cite{C}, such a complete hyperbolic surface is the complement of a finite set of points in a compact surface with boundary. Let $Y$ be the component of $S\setminus P$ containing $L$. Note that any simple closed geodesic $c$ in $Y$ intersects $L$ transversely (i.e. $L$ fills $Y$), otherwise we could add $c$ to $P$, contradicting maximality.
 \begin{figure}[htbp]
 	\includegraphics[scale=0.75]{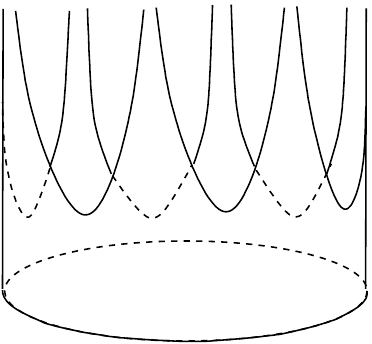}
 	\hspace{2cm} 
 	\includegraphics[scale=0.8]{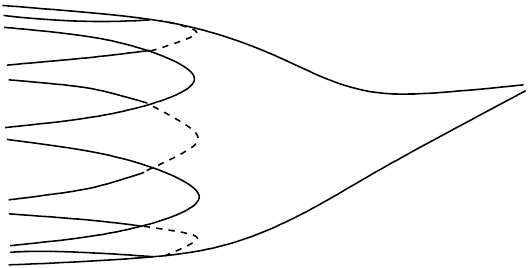} 
 	\caption{A crown (left) and a punctured crown (right).} 
 	\label{figure0} 
 \end{figure} 
 \begin{proposition} \label{crown}
 	Every component of $Y\setminus L$ is isometric to the interior of a finite sided ideal polygon, a crown, or a punctured crown.
 \end{proposition}

For closed surfaces, this basically follows from Lemma 4.4 of \cite{C}. Here we sketch an alternate proof for any surface. 

\begin{proof}[Sketch of proof.]
    By maximality of $P$, each component $Y'$ of $S\setminus (P\cup L)$ is a disk, an annulus, or a pair of pants. If $Y'$ is a disk, then it is the interior of a finite sided ideal polygon. If $Y'$ is an annulus, $Y'$ must be isometric to the interior of a crown or a punctured crown: Otherwise, it contains a simple closed geodesic, which contradicts the maximality of $P$. 
    
    Suppose $Y'$ is a pair of pants. We note that $Y$ must be the interior of a compact hyperbolic surface with closed geodesic boundary: Otherwise it contains a simple closed geodesic, again contradicting the maximality of $P$. Since $L$ contains no closed geodesics, $\overline{Y'}\cap L =\varnothing$, hence $Y'\nsubseteq Y$. So, every component of $Y\setminus L$ has the required type.
\end{proof}

Every geodesic lamination on $S$ consists of a finite set of minimal sublaminations together with a finite set of additional bi--infinite geodesics (isolated) where each end goes out a cusp of $S$ or is asymptotic to one of the minimal sublaminations; see \cite{Canary} or \cite{C}. If $L_0\subseteq L$ is a minimal component of a geodesic lamination $L$ on $S$ which is not a closed geodesic, we write $Y_{L_0}$ for the subsurface of $S$ filled by $L_0$ described above. If $L_0$ is a closed geodesic, let $Y_{L_0}$ be a small annular neighborhood of $L_0$.

 The set of all geodesic laminations on $S$ is denoted by $\mathcal{G}(S)$. The Hausdorff distance  $d_H$ between closed subsets of $S^0$ determines a metric on $\GS$ (any lamination $L$ is determined by $L\cap S^0$). This makes $\GS$ into a compact metric space. The notation $\xrightarrow{H}$ means convergence in this Hausdorff metric. 

A geodesic lamination $L$ is called an {\em ending lamination} if it is minimal and fills $S$; so every principle region is an ideal polygon or a punctured crown (which also refer to as a punctured ideal polygon). The set of all ending laminations is denoted by $\ELS$. We define another subset of $\mathcal{G}(S)$ called the {\em peripherally ending laminations} by 
\[\EL0S = \{\textit{L} \in \mathcal{G}(S) ~|~ \text{\textit{L} is minimal and fills a subsurface } Y_{L_0} \text{ containing all punctures}\}.\]
Note that  $\ELS \subseteq \EL0S \subseteq  \mathcal{G}(S)$, and for $L\in \EL0S$, every puncture is contained in a unique principle region which is a punctured ideal polygon; see Figure \ref{figure1}.

\begin{figure}[htbp]
	\includegraphics[scale=0.6]{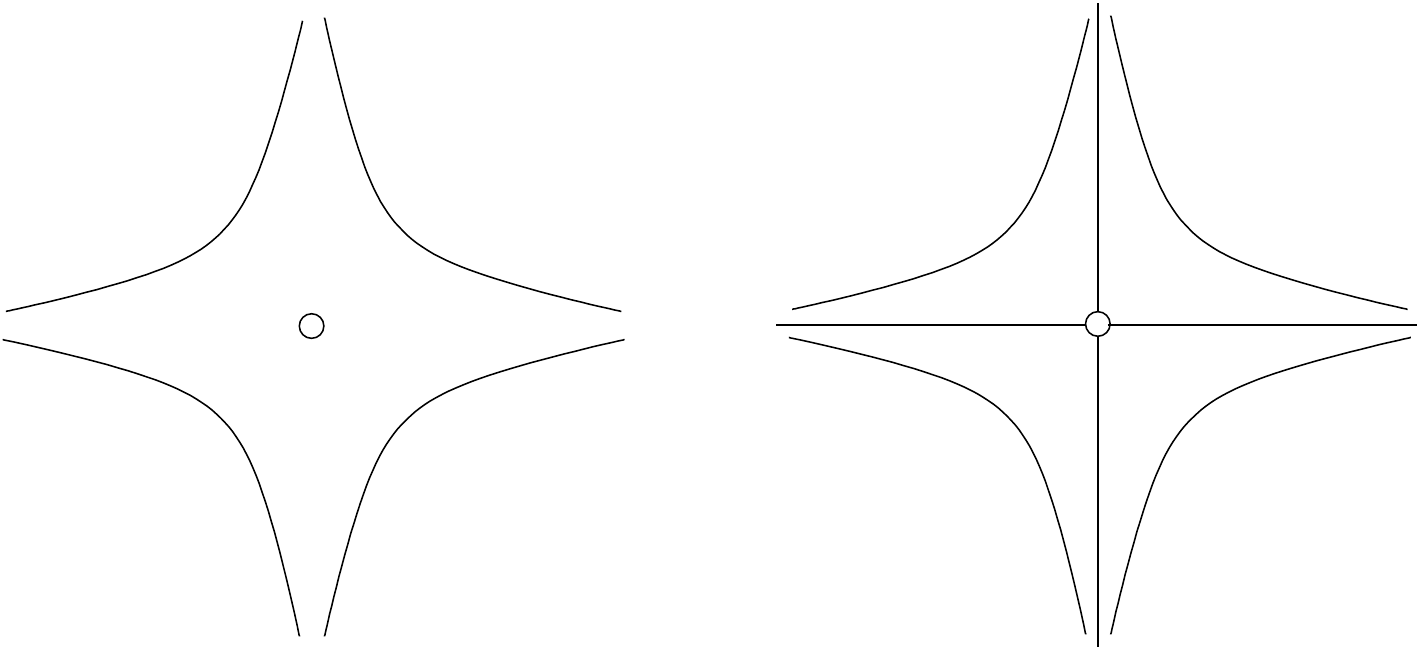} 
	\caption{A punctured crown (left) viewed as a punctured ideal polygon. All possible arcs from the puncture asymptotic to the sublamination (right).} 
	\label{figure1} 
\end{figure} 
 
  Next, we will describe the topology on $\EL0S$ and $\ELS$ that we will be interested in. Set \[U_\epsilon(L_0) = \{ L \in\EL0S \mid N_\epsilon(L) \supseteq L_0 \}\] where $N_\epsilon(L)$ is the $\epsilon$--neighborhood of $L$ on $S$.
   
  \begin{lemma} \label{subset}
  	If $L \in U_\epsilon(L_0)$, then there exists $\delta >0$ such that $U_\delta(L) \subseteq U_\epsilon(L_0)$.
  \end{lemma}
  \begin{proof} Assume that $L \in U_\epsilon(L_0)$. By definition, $L_0 \subseteq N_\epsilon(L)$. There is $0<\epsilon '< \epsilon $ such that $N_{\epsilon '} (L)\supseteq L_0$. Set $\delta = \epsilon - \epsilon '$ and let $L_1 \in U_\delta(L)$ so that $L \subseteq N_\delta(L_1)$ . Then \[L_0 \subseteq N_{\epsilon '} (L) \subseteq N_{\epsilon '}(N_{\epsilon - \epsilon '}(L_1))\subseteq N_{\epsilon}(L_1).\] This means $L_1 \in U_\epsilon(L_0)$, hence $U_\delta(L) \subseteq U_\epsilon(L_0)$, as required. 	
  \end{proof}
  
  Let $\mathcal{B} = \{U_\epsilon(L_0)| \epsilon >0$ and $L_0 \in \EL0S\}$. Since the elements in $\mathcal{B}$ cover $\EL0S$, Lemma \ref{subset} implies that $\mathcal{B}$ is the basis for a topology, and $\{ U_{\epsilon}(L_0)\}_{\epsilon >0}$ is a basis at $L_0$ (consequently, the topology is $1^{st}$ countable).
  
   For $\{L_n\}^{\infty}_{n=1}\subseteq \EL0S$, say $\{L_n\}$ {\em coarse Hausdorff converges to $L_0 \in \EL0S$}, written $L_n \xrightarrow{CH} L_0$, if for any subsequence $\{L_{n_k}\}$ such that $L_{n_k} \xrightarrow{H} L$, we have $L \supseteq L_0$; see \cite{Ham}. The next proposition tells us that convergence in the topology on $\ELS$ and $\EL0S$ just defined is precisely coarse Hausdorff convergence, and in particular, this is the {\em Thurston topology}; see Section 4.1 of \cite{Canary}.
 
  \begin{proposition} \label{seq}
  	$L_n \xrightarrow{CH} L$ if and only if $L_n \rightarrow L$ in the topology of $\EL0S$ defined above.
  \end{proposition}
  \begin{proof}
  	Assume that $L_n \xrightarrow{CH} L$. Suppose that $\{L_n\}$ does not converge to $L$. Then there exist $\epsilon > 0$ and a subsequence $\{L_{n_k} \}$ such that $L_{n_k} \notin U_\epsilon(L)$ for all $n_k$, i.e. $N_\epsilon(L_{n_k})\nsupseteq L$. By passing to a further subsequence if necessary, we may assume that $L_{n_{k}} \xrightarrow{H} L_0 \supseteq L$. By definition, there is $N>0$ such that for all $n_k>N$, $d_H(L_{n_k}, L_0)<\epsilon$. Therefore $L \subseteq L_0 \subseteq N_\epsilon(L_{n_k})$ which is a contradiction.
  	
  	Conversely, suppose that $L_n \rightarrow L$ in the above topology. Pass to any subsequence such that $L_{n_k} \xrightarrow{H} L_0$. Let $d_k' > d_H(L_{n_k},L_0)$ so that $d_k' \to 0$, and let $d_k''$ be such that $L \subseteq N_{d_k''}(L_{n_k})$, and so that $d_k'' \to 0$.  Now set $d_k = \max\{d_k',d_k''\}$. Observe that $L \subset \cap_k N_{d_k}(L_{n_k})$.  On the other hand, we can show that this intersection is exactly $L_0$.  To see this, first note that $L_0$ is contained each $N_{d_k}(L_{n_k})$ for all $k$, and so is contained in the intersection.  On the other hand, any point $x \in N_{d_k}(L_{n_k})$ has distance at most $2d_k$ to a point of $L_0$.  Therefore, the distance of any point $x$ in the intersection to a point in $L_0$ is zero, hence $x \in L_0$. It follows that $L\subseteq L_0$, and hence $L_n \xrightarrow{CH} L$.
  	
  \end{proof}
  \begin{corollary}
  	For any topological space $Y$, $f\colon \EL0S \to Y$ is continuous if and only if $f(L_n)\to f(L_0)$ whenever $L_n \xrightarrow{CH} L_0$.
  \end{corollary}

 \subsection{Unicorn arcs, unicorn paths and their properties}\label{unicorn}Given two arcs $a$ and $b$ that are in minimal position, choose an endpoint of $a$ and of $b$. A {\em unicorn arc between $a$ and $b$ } is an embedded arc obtained from a segment of $a$ from the endpoint and a segment of $b$ from the endpoint up to a point in $a \cap b$. Note that not all points in $a \cap b$ determine unicorn arcs. Given two unicorn arcs $a_i$ and $a_j$, we say that $a_i<a_j$ if $a_i$ contains a longer segment of $a$ than $a_j$. Let $\{a_1,a_2,...,a_{n-1}\}$ be the ordered set of all unicorn arcs. The sequence $P(a, b)=\{a=a_0,a_1,a_2,...,a_n=b\}$ is called the {\em unicorn path between $a$ and $b$.}  See \cite{A} for further details.
 
  The following two lemmas are used to prove that unicorn paths stay close to any geodesics connecting the endpoints in $\AS$; see Lemma \ref{closed1}. We will use them to prove a similar property in $\ACS$; see Lemma \ref{closed}.

 \begin{lemma} \cite{A} \label{subpath}
 	For every $0\leq i < j\leq n$, either $P(a_i,a_j)$ is a subpath of $P(a,b)=\{a=a_0,a_1,a_2,...,a_n=b\}$, or $j=i+2$ and $a_i$ and $a_j$ represent adjacent vertices of $\AS$.
 \end{lemma}
 \begin{lemma}\cite{A} \label{xxx}
 	Let $x_0,...,x_m$ with $m\leq 2^k$ be sequence of vertices in $\AS$. Then for any $c \in P(x_0, x_m)$, there is  $0\leq i < m$ with $c^* \in P(x_i,x_{i+1})$ at distance at most $k$ from $c$.                    
 \end{lemma}
 \begin{proposition}\cite{A}\label{closed1}
 	Given two arcs $a$ and $b$ in $\A(S)$ and $g$ a geodesic in $\AS$ connecting $a$ to $b$, every arc in $P(a,b)$ is within distance $6$ of $g$. Consequently, the Hausdorff distance between $g$ and $P(a,b))$ is at most 12.
 \end{proposition}
\begin{proof} The first statement is proved in \cite{A}. The last claim follows easily from this as we now explain. Consider consecutive points $x$ and $y$ of $P(a,b)$ and corresponding points $x'$ and $y'$ of $g$ with distance at most $6$ from $x$ and $y$, respectively. By the triangle inequality, we have $d(x',y') \leq d(x',x) + d(x,y) + d(y,y') = 6+1+6=13.$  Thus any point in $g$ between $x'$ and $y'$ is distance at most $6$ from one of $x'$ or $y'$, and hence the distance to one of $x$ or $y$ is at most $12$. If we consider all pairs of consecutive points $x$ and $y$ of $P(a,b)$ and all corresponding pairs of points $x'$ and $y'$ in $g$, union of all subpaths of $g$ connecting such pair $x'$ and $y'$ covers $g$. This together with Proposition \ref{closed1} imply that $d_H(g,P(a,b))\leq 12$.
\end{proof}
\begin{lemma} \label{closed}
	Given two arcs $a$ and $b$ in $\ACS$ and $g$ a geodesic in $\ACS$ connecting $a$ and $b$, then every curve in $P(a,b)$ is within distance $7$ of $g$. Consequently, the Hausdorff distance between $g$ and $P(a,b)$ is at most $14$.
\end{lemma}
\begin{proof}
	  Let $c\in P(a,b)$ be at maximal distance $k>0$ from $g$. Let $\bar a' \bar b'$ be the maximal subpath of  $P(a ,b)$ containing $c$ with $\bar a'$ and $ \bar b'$ at distance $2k$ from $c$. If no $\bar a'$ exists, then $d(c,a) < 2k$, and we set $\bar a' = a$, and similarly for $\bar b'$. Then, by Lemma \ref{subpath}, $P(\bar a' ,\bar b') \subseteq P(a , b)$. Let $a' $ and $ b' $  be  vertices on $g$ closest to $\bar a'$ and $\bar b'$, respectively. In the case when $\bar a'=a$ and/or $\bar b'=b$, let $a'=a$ and/or $b'=b$, respectively.  We have $d(a' ,\bar a') \leq k$ and $d(b',\bar b') \leq k$. Thus $d(a', b')\leq 6k$. Concatenate the geodesic segment $a'b'$ of $g$ with any geodesics paths $a'\bar a'$ and  $b'\bar b'$. Let $\bar a'=x_1, x_2, ..., x_m = \bar b'$ be the consecutive vertices of the concatenation where $m\leq 8k$. For $1 \leq i\leq m-2$, let $\bar x_i$ be an arc adjacent to both $x_i$ and $x_{i+1}$. By Lemma \ref{xxx}, $c$ is at distance $\leq \lceil \log_2 8k-1 \rceil+1$ from some $x_i$. If $x_i \in g$, then $k \leq \lceil \log_2 8k-1 \rceil+1$. Otherwise, if $x \notin g$, x $\in$ $a'\bar a'$ or  $b'\bar b'$. Since $d(c,x_i)\geq d(c,\bar a')-d(a' ,\bar a') \geq k$, we also have  $k \leq \lceil \log_2 8k-1 \rceil+1 $. Thus $k\leq 7$.
\end{proof}	

 \section{ARC GRAPH}

 \subsection{Infinite unicorn paths}  Assume that $S$ has at least one puncture. Fix a puncture and let $a$ be an arc in $\AS$ realized by its geodesic representative in $S$ whose ends are at the puncture. Consider $L_0\in \EL0S$ and $l$ a bi--infinite geodesic starting at the puncture asymptotic to $L_0$. Note that $l \cap L_0 = \varnothing$. To choose an endpoint of $a$, we fix an orientation for $a$ so the terminal point is the chosen endpoint and consider the puncture as the endpoint of $l$.  A {\em unicorn arc for $a$ and $l$} is a simple arc consisting of a segment of $a$ and a segment of $l$ from the endpoints up to a point of intersection. For any two distinct unicorn arcs $a_i$ and $a_j$ constructed from $a$ and $l$, we say that $a_i<a_j$ if $a_i$ contains a longer segment of $a$ than $a_j$.  We consider all unicorn arcs from $a$ and $l$ in order and write this as $\{a=a_0, a_1, a_2,...\}=\{a_n\}=P(a,l)$. We call this {\em the infinite unicorn path defined by $a$ and $l$}. Set $\{x_1,x_2,...\}\subseteq a \cap l$ to be the set of intersection points corresponding to each unicorn, appearing in order along $a$. We write $a_i=a^{\circ}_i\cup l^{\circ}_i$ where $a^{\circ}_i \subseteq a_i$ and $l^{\circ}_i \subseteq l_i$ are rays (i.e. subarcs) and $a^{\circ}_i\cap l^{\circ}_i = x_i$. For each $i$, we will use $a_i$ to denote both the arc consisting of the subarcs $a_i^{\circ}$ and $l_i^{\circ}$	as well as its isotopy class, and its geodesic representative, with context clarifying the meaning. When necessary, we will use different notation.

\begin{proposition} For any arc $a$, $L_0 \in \EL0S$, and $l$ asymptotic to $L_0$, $P(a,l)$ contains infinitely many arcs.
\end{proposition}
\begin{proof}  The last point of intersection $z$ of $a$ with $L_0$ is at a boundary leaf which is one side of a punctured ideal polygon (since $L_0 \in \EL0S$). Observe that $a$ cannot intersect $l$ after $z$ (compare with Figure \ref{figure1}). However, the points of intersection $a \cap l$ must accumulate on $z$ since $l$ is asymptotic to $L_0$ and any leaf of $L_0$ is dense hence $l$ is dense in $L_0 \cup l$. So, given $a_i \in P(a,l)$ defined by $x_i \in a \cap l$, the next time $l$ intersects the arc of $a$ between $x_i$ and $z$ is the point $x_{i+1}$, and hence $a_{i+1}$ is defined.  Since $i$ was arbitrary, this completes the proof.
\end{proof}

The way we define infinite unicorn paths $P(a,l)$ is also valid for any lamination $L$ and any geodesic $l$ asymptotic to $L$.  However, we cannot guarantee that $P(a,l)$ will contain infinitely many arcs in general.

For the next lemma, recall our convention about our parameterizations of geodesics; see Section \ref{arc}
\begin{lemma} \label{star1}
	Let $a \in \AS$. Given $\epsilon >0$ and $R>0$, there is $N>0$ such that for any $L \in \EL0S$, if $l$ is asymptotic to $L$ and $P(a ,l)=\{a_0, a_1,...\}$, then as parametrized geodesics $a_i(t)$ and $l(t)$, we have $d(a_i(t),l(t))<\epsilon$ for all $t\in (-\infty,R]$ and for all $i\geq N$.
\end{lemma}
\begin{proof}
	Since $a\cap S^0$ is a compact arc, there is $\epsilon '>0$, so that the $\epsilon '$--neighborhood of $a$ in $S^0$, $N_{\epsilon'}(a\cap S^0) \subset S^0$, is a tubular neighborhood homeomorphic to $(a\cap S^0)\times [-\epsilon ',\epsilon ']$. Observe that the angle of intersection between $L$ and $a$ has a lower bound $\theta_0$ where $\theta_0$ depends only on $a$. If not, some $L_0\in \EL0S$ has a leaf $l_0$ intersecting $a$ at such a small angle that $l_0\cap S^0 \subseteq N_{\epsilon '}(a\cap S^0)$. Then $l_0=a$, contradicting $L_0\in \EL0S$. 
	
	Now the distance between consecutive points of intersection $l\cap a$ is bounded below by $2\epsilon '$, so if $0<t_1<t_2<...$ are such that $l(t_i)=x_i \in l \cap a$, the intersection point defining $a_i$, then $t_i>(2\epsilon ')(i-1)$. Let $a_i^*$ be the geodesic representative of $a_i = l(-\infty,t_i] \cup a_i^\circ = l_i^\circ \cup a_i^\circ$.  Since the angle of intersection is greater than $\theta_0$, there are lifts $\widetilde a_i^*$ of $a_i^*$ and $\widetilde a_i$ of $a_i$ to the universal cover that have uniformly bounded Hausdorff distance (with the bound depending only on $a$).  In particular, there is a constant $K > 0$ (depending only on $a$) such that $d(\widetilde a_i^*(t),\widetilde a_i(t)) \leq K$, for all $t \in (-\infty,t_i)$.  On the other hand, the lift $\widetilde l$ of $l$ agreeing with $\widetilde a_i$ on its initial segment has $d(\widetilde a_i^*(t),\widetilde l(t)) < e^t \delta$ for $t < 0$, where $\delta$ is the length of $\partial S^0$.  Since $t_i > (2 \epsilon')i$, there is an $N > 0$ so that if $i > N$, we have $t_i >> R$, and hence by convexity of the hyperbolic distance function, we have $d(\widetilde a_i^*(t),\widetilde l(t))< \epsilon$ for all $t \in (-\infty,R]$.  Consequently, $d(a_i^*(t),l(t)) < \epsilon$ for all $t \in (-\infty,R]$.

\end{proof}	
\begin{corollary} \label{accum}
	If $L \in \EL0S$ and $l$ is asymptotic to $L$, then any Hausdorff accumulation point of the sequence $P(a,l)$ contains $l$, and hence $L$.
\end{corollary}

\subsection{Construction of a continuous map}
Here we use infinite unicorn paths to construct a continuous map from $\EL0S$ to $\DAS$.  In the next two lemmas, we assume $a$ is an arc, $L_0 \in \EL0S$ and $l$ is a simple geodesic asymptotic to $L_0$.
  
 \begin{lemma} \label{subpath1}
 	Infinite unicorn paths restrict to finite unicorn paths. More precisely, if $a_j \in P(a,l)$ and $j \geq 3$,  then $P(a,a_j)\subseteq P(a,l)$.
 \end{lemma}
 Here $P(a,a_j)$ is a unicorn path as in Section \ref{unicorn}.
 \begin{proof}
 	Let $P(a,l)=\{a_0,a_1,a_2,...\}$, realizing each $a_i$ by the geodesic representative of $l_i^{\circ}\cup a_i^{\circ}$, and let $x_i=l_i^{\circ}\cap a_i^{\circ}$. Assume that $a_j \in P(a,l)$ with $j\geq 3$. By Lemma \ref{star1}, there is $m\gg j$ such that $a_m$ is close to $l$ for all intersection points of $l$ with $a$ up to $x_j$. Then the first $j+1$ points of $P(a,a_m)$ are exactly $a_0,a_1,...a_j$. By Lemma \ref{subpath}, $P(a,a_j) \subseteq P(a,a_m)$, so $P(a,a_j) = \{a_0,a_1,\ldots,a_j\}$. Thus  $P(a,a_j)\subseteq P(a,l)$ as required. 
 \end{proof}

 The next lemma is similar to the proof that the curve graph has infinite diameter given in \cite{Hempel}.
\begin{lemma} \label{infinity} $\displaystyle{\lim_{n \to \infty} d(a,a_n)= \infty}$ where $\{a_n\}=P(a,l)$.
\end{lemma}
\begin{proof} 
	 To prove the lemma, suppose for a contradiction that $\displaystyle{\lim_{n \to \infty} d(a,a_n)\neq \infty}$. By Proposition \ref{closed1} and Lemma \ref{subpath1}, $d(a,a_n)\leq d(a,a_m)+6$ for all $m>n$, so $\sup d(a,a_n) < \infty$. Then there is some $N>0$ and an infinite subsequence $\{a_n\}$ with $d(a,a_n )=N$. By Corollary \ref{accum}, we may pass to a further subsequence $\{a_n\}$ so that $a_n \xrightarrow{H}$ $L$ with $L\supseteq L_0 \in \EL0S$. For each $n$, we have $a^1_n$ with $d(a_n,a^1_n )=1$ and $d(a, a^1_n )=N-1$. We may assume that $a^1_n \xrightarrow{H} L^1$ where $L^1$ is a lamination (pass to a subsequence if necessary). Since $d(a_n,a^1_n )=1$, $L\pitchfork L^1 =\varnothing$, and so $L_0\pitchfork L^1=\varnothing$. Since $L_0\subseteq L$ and $L_0$ is minimal and fills $Y_{L_0}$, a subsurface containing all punctures, a leaf of $L^1$ intersects $Y_{L_0}$. Thus the leaf has to be a leaf of $L_0$ or asymptotic to $L_0$. These facts imply that $L^1\supseteq L_0$. Proceeding inductively, for each $k=1,\ldots,N$ we get sequences $\{a_{n}^k\}_{n=1}^{\infty}$ so that $d(a,a_{n}^k)=N-k$, $a_{n}^k \xrightarrow{H} L^k$, and $L^k \supseteq L_0$. But $a_n^N=a$ for all $n$,  a contradiction.      
\end{proof} 

For arcs $a$ and $b$, a geodesic in $\AS$ connecting $a$ and $b$ is denoted by $[a,b]$. The following Proposition tells us that for $L \in \EL0S$, $P(a,l)=\{a_n\}$ defines a point in $\DAS$ which we denote $[P(a,l)]\in \DAS$.
\begin{proposition} \label{point}
	Let $L\in \EL0S$ and $l$ be a geodesic ray asymptotic to $L$. Then $P(a,l)=\{a_n\}$ defines a point in $\DAS$. Moreover, for any two geodesic rays $l$ and $l'$ asymptotic to $L$, we have $[P(a,l)] = [P(a,l')] \in \partial\AS$.	
\end{proposition}
\begin{proof}
	 For any $R > 0$, Lemma \ref{infinity} gives $N >0$ such that $d(a,a_n)>R$ for all $n\geq N$. For all $m,n \geq N$, we have $(a_n,a_m)_a \geq d(a,[a_n,a_m])-2\delta$. Since $[a_n,a_m]$ and $P(a_n,a_m)$ have Hausdorff distance at most 12, by Proposition \ref{closed1}, this implies that $(a_n,a_m)_a\geq d(a,P(a_n,a_m))-12-2\delta \geq R-12-2\delta$. For $|m-n|>2$, $P(a_n,a_m)$ is contained in $P(a,l)$ by Lemma \ref{subpath1}, so  $[P(a,l)] \in \DAS$.
	
	 It remains to show the latter part. First note that $l,l'$ are disjoint. Let $a_i \in P(a,l)$. We write $a_i = a_i^{\circ}\cup l^{\circ}_i$. Since $L$ is minimal, $l' \cap a_i^{\circ} \neq \varnothing$. If we parametrize $l'$, the first time $l'$ intersects $a_i^{\circ}$ defines a unicorn arc in $P(a,l')$ disjoint from $a_i$. Similarly, for each point in $P(a,l')$, we can find a point in $P(a,l)$ disjoint from it. Consequently, the Hausdorff distance between $P(a,l)$ and $P(a,l')$ is one which finishes the proof.       
\end{proof}
\begin{proposition} \label{continous}
	Consider the map
	\[ F \colon \EL0S \to \partial \A(S) \]
	defined by $F(L) = [P(a,l)]$ where $l$ is any geodesic asymptotic to $L$. Then $F$ is continuous.
	
\end{proposition}
\begin{proof}
	Let $\{L_k\}$ be a sequence of laminations in $\EL0S$ and $L_0 \in \EL0S$ such that $L_k \to L_0$. By Proposition \ref{seq}, $L_k \xrightarrow{CH} L_0$. Let $\{l_k\}$ be a sequence of bi-infinite geodesics with $l_k$ asymptotic to $L_k$ for each $k$. Then each $l_k$ intersects a small compact circle of around the cusp, so up to subsequence, $l_{k_j} \to l$ as parametrized geodesics. Since $L_0\in \EL0S$ and $l$ is asymptotic to $L_0$, $\{l_k\}$ is a union of finitely many convergent subsequences (see Figure \ref{figure1}). Any Hausdorff limit of any subsequence of $\{L_k\}$ contains $L_0$, which fills a subsurface $Y_{L_0}$ containing all punctures. Since $l$ must intersect $Y_{L_0}$ and have no transverse intersection with $L_0$, it follows that $l$ is asymptotic to $L_0$. This means $\{l_k\}$ spits into finitely many convergent subsequences. Since $l_{k_j}$ limits to $l$, it follows that $P(a,l_{k_j})$ and $P(a,l)$ agree on longer and longer initial intervals, hence $F(L_{k_j})=[P(a,l_{k_j})]\to [P(a,l)]=F(L_0)$ (this follows from Proposition \ref{closed1}, Lemma \ref{subpath1}, and the fact that every geodesic triangle is thin). This holds for any of the finitely many subsequences $\{L_{k_j}\}$ with $l_{k_j} \to l$ for some $l$ as a parametrized geodesic and hence $F(L_k) \to F(L_0)$.     
\end{proof}

\subsection{Homeomorphism and Theorem \ref{main}}
Now that we have constructed a continuous map $F \colon \EL0S \to \DAS$, we set about proving that it is a homeomorphism.  We begin with the proof of injectivity of $F$.

 \begin{lemma} \label{inject}
 	The map $F \colon \EL0S \to \partial A(S)$ is an injection.
 \end{lemma}
 
 \begin{proof}
  Let $L_1\neq L_2$ in $\EL0S$. Set $l_1$ and $l_2$ to be bi--infinite geodesics asymptotic to $L_1$ and $L_2$, respectively. Then we have $|l_1 \cap l_2 |=\infty$. Parametrize $l_1$ and $l_2$ (recall our convention on parametrization of geodesic) and let $t$ be the smallest real number such that $l_1([-\infty,t]) \cap l_2([-\infty,t]) \neq \varnothing$.  Let b be the arc defined by segments of $l_1$ and $l_2$ up to a point in $l_1([-\infty,t]) \cap l_2([-\infty,t])$ (if there are two such points, pick one). Let $P(a,l_i)=\{a^i_j\}^{\infty}_{j=1}$, $i=1,2.$ By Lemma \ref{star1}, $a^i_j$ stays close to $l_i$  for a very long time, for each $i=1,2.$ In particular, it follows that for all sufficiently large $n$ and $m$, $b$ is in $P(a^1_m,a^2_n)$. Therefore, the geodesic from $a^1_m$ to $a^2_n$ passes within distance $6$ of $b$ for all $n$ and $m$ that are sufficiently large. Hence $(a_m^1,a_n^2)_a \leq d(a,b) + 6 + 2 \delta$, so that $[P(a,l_1)] \neq [P(a,l_2)]$.
 	
 \end{proof}	
 
 To show that F is also surjective, we need the following lemma.
 
\begin{lemma} \label{ont}
	If $[\{c_n\}] \in \DAS$, then $c_n\xrightarrow{CH} L_0$ and $L_0\in \EL0S$ with $F(L_0)=[\{c_n\}]$.
\end{lemma}
\begin{proof}
	Let  $\{c_n \}$ be a sequence in $\AS$ that defines a point in $\DAS$. Suppose $\{c_n\}$ is any subsequence Hausdorff converging to a lamination $L$. We may assume $c_n \to l$ as parameterized geodesics, where $l\subseteq L$. Let $L'$ be the derived lamination of $L$. If there is a component $L_1\subseteq L'$ filling a subsurface $Y_{L_1}$ containing all the punctures, then $l$ is asymptotic to $L_1$ since $l$ has one end at a puncture. Suppose there is no such component of $L'$. The geodesic $l$ is asymptotic to some component $L_0\subseteq L'$ filling a subsurface $Y_{L_0}$, and by assumption $Y_{L_0}$ cannot contain all the punctures.

	By assumption, there exists an arc $a$ outside $Y_{L_0}$ such that $|a\cap l | < \infty$.  Indeed, there is an initial subarc $l_0 \subseteq l$ so that $l \setminus l_0 \subseteq Y_{L_0}$, and hence $a \cap l = a \cap l_0$.  Since $c_n \to l$, there is an $N  > 0$ so that for all $n \geq N$, $c_n$ has an initial arc $c_n^0$ so that $c_n^0$ is isotopic to $l_0$ in $S \setminus Y_{L_0}$.  Hence, for all $n\geq N$, $P(a,c_n)\supseteq P(a,l)$.  For each $n\geq N$, the arc $c_n$ returns to the cusp after entering $Y_{L_0}$, so must intersect $l$ at some point, necessarily in $Y_{L_0}$ before leaving $Y_{L_0}$. Thus, there is $m_n>n$ so that $c_{m_n}$ follows $l$ closely until this point of intersection, and hence $P(c_n,c_{m_n})$ contains an arc $b_n$ built from subarcs of $c_n$ and $c_{m_n}$ whose respective intersections with $S \setminus Y_{L_0}$ are precisely $c_n^0$ and $c_{m_n}^0$.  Therefore, $|a \cap b_n| \leq 2 |a \cap l |$.  This gives a uniform distance from $P(c_n, c_{m_n})$ to $a$ for all $n>N$. Therefore from Section 2.1 and Proposition \ref{closed1}, we have
	\[ (c_n,c_{m_n})_a \leq d(a,[c_n,c_{m_n}]) + 2 \delta  \leq d(a,P(c_n,c_{m_n})) + 2 \delta + 12  \]
	which contradicts the fact $[\{c_n\}] \in \DAS$. Hence, $L_0 \in \EL0S$. Fix any arc $a$.  In any subsequence as above with $n$ sufficiently large, $c_n$ and $l$ are very close on long initial segments.  Consequently, $P(a,l)$ and $P(a,c_n)$ are agree on long initial intervals.  It follows that $c_n \to [P(a,l)] = F(L_0)$, hence $[\{c_n\}]=F(L_0)$. Since we passed to an arbitrary Hausdorff convergent subsequence and $F$ is injective, we have $c_n \xrightarrow{CH} L_0$.
	
\end{proof}

By Lemma \ref{inject} and \ref{ont}, we immediately have the next proposition.

\begin{proposition} \label{bijec}
	The map $F \colon \EL0S \to \partial A(S)$ is a bijection.
\end{proposition}

 	Next we show that $F^{-1}$ is continuous.
	
\begin{lemma} \label{onto}
	Let $\{L_n\}^{\infty}_{n=1}$ be a sequence in $\EL0S$. If $F(L_n) \rightarrow F(L_0)$ in $\partial\AS$, then $L_n \xrightarrow{CH} L_0$.
\end{lemma}
\begin{proof} For all $n\geq 0$, set $F(L_n)=[P(a,l_n)]=[\{a_{i,n}\}^{\infty}_{i=0}]$ where $l_n$ is asymptotic to $L_n$. Consider a Hausdorff convergent subsequence $L_n \xrightarrow{H} L$ where $L$ is a lamination. By passing to a further subsequence we may suppose that $l_n \rightarrow l$ where $l$ is asymptotic to $L$. Since $F(L_n) \rightarrow F(L_0)$ in $\DAS$, for any $r>0$, there is $n_r$ such that $a_{j,n_r} \in N_{2\delta + 12}(\{a_{i,0}\}_i)$ for all $j$ with $d(a,a_{j,n_r}) \leq r$ (we are using the fact that subsegments of unicorn paths have Hausdorff distance at most 12 from geodesics connecting their endpoints). For each $r > 0$, pick $i_r> 0$ so that $d(a,a_{i_r,n_r}) = r$, and consequently $a_{i_r,n_r} \in N_{2 \delta + 12}(\{a_{i,0}\})$.  For any $R > 0$ and $\epsilon > 0$, Lemma 3.2 guarantees that for $r$ sufficiently large, $d(a_{i_r,n_r}(t),l_{n_r}(t))< \epsilon$ for all $t \in (-\infty,R]$.  On the other hand, $l_n \to l$ as parameterized geodesics.  Therefore, $a_{i_r,n_r} \to l$ as $r \to \infty$, also as parameterized geodesics.  Since $[\{a_{i_r,n_r}\}]= F(L_0)$, by Lemma \ref{ont} the closure of $l$ contains $L_0$. Since $l$ is asymptotic to $L$, $L\cup l$ is a lamination containing $l$, and since $l\nsubseteq L_0$, $L_0\subseteq \bar{l}\setminus l \subseteq (L\cup l)\setminus l = L$.
\end{proof}

\begin{proof}[Proof of Theorem \ref{main}]
	That $F$ is a homeomorphism follows immediately from Proposition \ref{continous},  \ref{bijec}, and Lemma \ref{onto}. Furthermore, if $\{a_n\} \in \AS$ is a sequence converging to $F(L_0)$, by Lemma \ref{ont}, any Hausdorff accumulation point of $\{a_n\}$ in $\mathcal{G}(S)$ contains $L_0$.
	
	To see that $F$ is $Mod(S)$--equivariant, note that for any $f \in Mod(S)$ and point $[\{c_n\}] = F(L_0)$, the Hausdorff accumulation points of $\{f(c_n)\}$ are precisely the $f$--image of the Hausdorff accumulation points of $\{c_n\}$, and hence all contain $f(L_0) \in \EL0S$.  Thus, by the first part, it follows that $f(F(L_0)) = f([\{c_n\}]) = [\{f(c_n) \}] = F(f(L_0))$, as required
\end{proof}

\section{ARC AND CURVE GRAPH}

In this section, we prove Theorem \ref{main1}. We first use the same technique to prove Theorem \ref{main} when $S$ is a punctured surface. Then we use the result for the punctured surfaces to prove Theorem \ref{main} for the case when $S$ is a closed surface.  Note that $\ELS\subseteq \EL0S$. It follows that some results in Section 2 can be used in this section.
\subsection{Punctured surface} Assume that $S$ is a connected hyperbolic surface of finite area with finitely many punctures. We observe that if $l$ is asymptotic to $L\in \ELS$, $P(a,l)$ represents a point in the Gromov boundary, and this can be used to define a continuous map. The notation $[P(a,l)]$ is still used to distinguish between the path $P(a,l)$ and the point in the boundary. The next two propositions are analogous to Lemma \ref{infinity} and Proposition \ref{point} and \ref{continous}. The proofs are essential identical, so we omit them.
\begin{proposition} \label{point1}
	Let $L\in\ELS$ and $l$ be a simple geodesic asymptotic to $L$. Then $P(a,l)=\{a_n\}$ defines a point in $\DACS$. Moreover, for any two geodesic rays $l,l'$ asymptotic to $L$, we have $[P(a,l)] = [P(a,l')] \in \partial\ACS$.	
\end{proposition}

\begin{proposition} \label{continous1}
	Consider the map
	\[ F \colon \ELS \to \partial \ACS \]
	defined by $F(L) = [P(a,l)]$ where $l$ is any geodesic asymptotic to $L$. Then $F$ is continuous.
	
\end{proposition}

We note here that $F$ is injective (this follows directly from the arguments of Proposition \ref{bijec} combining with Lemma \ref{closed}). The next lemma mimics Lemma \ref{ont}. The proof is slightly different, so we have included  the relevant details.
\begin{lemma} \label{ont1}
	If $[\{b_n\}] \in \DACS$, then $b_n\xrightarrow{CH} L_0$ and $L_0\in \ELS$ with $F(L_0)=[\{b_n\}]$.
\end{lemma}
\begin{proof}
	Let  $\{b_n \}$ be a sequence in $\ACS$ that defines a point in $\DACS$ and $\{c_n \}$ be a sequence $\AS$ such that $c_i$ is adjacent to $b_i$ for all $i$. Then $[\{c_n \}]$ is also a point in $\DACS$ with $[\{b_n\}]=[\{c_n\}]$. We may pass to a subsequence to get $c_n \xrightarrow{H}$ L where $L$ is a lamination. We will first show that $L \supseteq L_0 \in \ELS$. Suppose for a contradiction that $L'$, the derived lamination of $L$, is not an ending lamination.
	
	 As parametrized geodesics, $c_n \rightarrow l\subseteq L$ up to subsequence where $l$ is a geodesic asymptotic to $L_1\subseteq L'$. Since $L'$ is not an ending lamination, $Y_{L_1}$ is not $S$ (see Section \ref{lamination} for discussion on the structure of laminations). Then there exists an essential simple closed curve $a$ in $S\setminus Y_{L_1}$ such that $|a \cap l| < \infty$. We can use this $a$ as in the proof of Lemma \ref{ont} and get a contradiction in the same way, hence $L_0=L' \in \ELS$. Similar to the proof of Lemma \ref{ont}, we have $c_n\xrightarrow{CH} L_0$ and $F(L_0) =[P(a,l)]=[\{c_n\}]=[\{b_n\}]$. Since $c_n$ and $b_n$ have no transverse intersection, any Hausdorff limit of $\{b_n\}$ has no transverse intersection with $L_0$, hence contains $L_0$. Therefore $b_n\xrightarrow{CH} L_0$.

\end{proof}

\begin{proof}[Proof of Theorem \ref{main1}]
	 The map $F$ given by Proposition \ref{continous1} is surjective by Lemma \ref{ont1}. Continuity of $F^{-1}$ follows the same basic argument as in Lemma \ref{onto}. Also, $F$ is $Mod(S)$--equivariant since for any $f \in Mod(S)$ and point $[\{b_n\}] = F(L_0)$,  $f(F(L_0)) = f([\{b_n\}]) = [\{f(b_n) \}] = F(f(L_0))$ as in the proof of Theorem \ref{main}.
	 
\end{proof}

\subsection{Closed surface}

In this section, we show that the Gromov boundary of $\CS$ is the space $\ELS$ when $S$ is a closed surface. Consider a hyperbolic metric $m_0$ on $S$. According to \cite{Birman}, the set of simple geodesics on $(S,m_0)$ is nowhere dense. Then we can find a disk neighborhood $D$ around a point $x$ in $S$ which is disjoint from all geodesic laminations on $S$. Next, we use metric interpolation to modify the metric $m_0$ on $S\smallsetminus x$ to a metric $m_1$ which is complete, pinched negatively curved, and so that:\begin{enumerate}
	\item $m_1=m_0$ on $S\smallsetminus D$,
	\item in a neighborhood of $x$ in $D\smallsetminus x$, $m_1$ is hyperbolic.
\end{enumerate} This is an explicit calculation in polar coordinates about $x$. The same calculation in 3--dimensions is attributed to Kerckhoff and appears in the proof Theorem 1.2.1 of \cite{kojima}. 

Now, we realize every simple closed curve on $S$ as an $m_0$--geodesic and note that they are also $m_1$--geodesics on $S \smallsetminus x$. Hausdorff convergence in $\mathcal{G}(S,m_0)$ (that is, using the metric $m_0$) of any sequence of such geodesics is the same as Hausdorff convergence in $\mathcal{G}(S\smallsetminus x,m_1)$. Hence $(\mathcal{G}(S),m_0)$ and $(\mathcal{EL}(S),m_0)$ embed as closed subsets of $(\mathcal{G}(S\smallsetminus x),m_1)$ and $(\mathcal{EL}(S\smallsetminus x),m_1)$, respectively. For the next lemma, let $m_2$ be any complete hyperbolic metric on $S\smallsetminus x$.

\begin{lemma} There is a bi--Lipschitz homeomorphism $f:(S\smallsetminus x,m_1) \to (S\smallsetminus x,m_2)$ isotopic to the identity on $S\smallsetminus x$, which is an isometry on some cusp neighborhood and $f$ lifts to a quasi--isometry $\bar{f}$ of the universal covers.
\end{lemma}

\begin{proof} We first isotope the identity so that it is a diffeomorphism with respect to the smooth structure for $m_1$ and for $m_2$.  Next, note that any two hyperbolic cusps contain possibly smaller cusps which are isometric.  Now after an isotopy, we can assume that the diffeomorphism $f:(S\setminus x,m_1)\to (S\smallsetminus x, m_2)$ is an isometry on some cusp neighborhood.  Since the complement of the cusp is compact, there is a bound on the bi--Lipschitz constant of the derivative, and hence the map is K--bi--Lipschitz for some $K>1$.  So, $f$ increases lengths of curves by at most a factor of $K$ and decreases them by a factor of at worst $1/K$. Since the pull back metric on the universal covers are path metrics so that the universal covering is a local isometry, this means that lengths of paths in the universal cover are distorted by at worst $K$ and $1/K$. This implies that distances are also distorted by at worst $K$ and $1/K$, so the lift of $f$ is a bi--Lipschitz in the universal covering, hence it is a quasi--isometry.
\end{proof}

Since laminations, ending laminations, and Hausdorff convergence can be defined in terms of the circle at infinity of the universal covering, the lemma proves:
\begin{corollary} \label{homeo}
	There is a homeomorphism $f': \mathcal{G}(S\smallsetminus x, m_1) \to \mathcal{G}(S\smallsetminus x, m_2)$ which induces a homeomorphism from $\mathcal{EL}(S\smallsetminus x, m_1)$ to $\mathcal{EL}(S\smallsetminus x, m_2)$. 
\end{corollary}  

 We know that the realization by geodesics defines an isometric embedding of $\CS$ into $\mathcal{C}(S\smallsetminus x)$; see \cite{A}.  This also realizes $\partial\CS$ as a subset of $\partial\mathcal{C}(S\smallsetminus x)$.

\begin{lemma} \label{closed2}
The isometric embedding $\CS \to \mathcal{C}(S\smallsetminus x)$ induces an embedding $\partial\CS \to \partial\mathcal{C}(S\smallsetminus x)$ onto a closed subspace.
\end{lemma}
\begin{proof}
	Let $\{[\{c_{n,k}\}_{n=1}^{\infty}]\}_{k=1}^{\infty}$ be a sequence of points in $\partial\CS$ converging to $[\{c_n\}_{n=1}^{\infty}] \in \partial\mathcal{C}(S\smallsetminus x)$. We want to show that $[\{c_n\}_{n=1}^{\infty}] \in  \partial\CS$. For each $r>0$, there are $k_r$ and $n_r$ such that $(c_{n_{k_r},k_r}\cdot c_{n_{r}})_o >r$. Thus $[\{c_{n_{k_r},k_r}\}_{r=1}^{\infty}]=[\{c_n\}_{n=1}^{\infty}]$. Since $c_{n_{k_r},k_r} \in \CS$ for all $r$, $[\{c_n\}_{n=1}^{\infty}] = [\{c_{n_{k_r},k_r}\}_{r=1}^{\infty}] \in \partial\CS$.
\end{proof}

We now identify $\GS$, $\ELS$, and $\partial\mathcal{C}(S)$ with their respective images in $\mathcal{G}(S\smallsetminus x)$, $\mathcal{EL}(S\smallsetminus x)$, and $\partial\mathcal{C}(S\smallsetminus x)$.  Let $F:\mathcal{EL}(S\smallsetminus x)\to \partial\mathcal{C}(S\smallsetminus x)$ be the homeomorphism from Theorem \ref{main1}, already proved in the punctured case. Suppose $[\{c_n\}]$ is a point in $\partial\CS$. This is also a point in $\partial\mathcal{C}(S\smallsetminus x)$, so any Hausdorff accumulation point of $\{c_n\}$ contains $L \in \mathcal{EL}(S\smallsetminus x)$ where $F(L)=[\{c_n\}]$.  On the other hand, any Hausdorff accumulation point of $\{c_n\}$ is in $\GS$ since $\GS$ is closed in $\mathcal{G}(S \smallsetminus x)$. Since $L$ is in $\mathcal{EL}(S\smallsetminus x)$, every leaf of $L$ is dense and all complementary regions are ideal polygons or once--punctured ideal polygons, so in fact $L$ is in $\ELS$. Let $\varOmega= F^{-1}(\partial\CS)$ which is a closed subset of $\ELS$. If $[\{c_n\}] = F(L)$, then any Hausdorff accumulation point of $\{c_n\}$ contains $L$. Let $F'=F|_{\varOmega}:\varOmega\to \partial\CS$ be the restricted homeomorphism. 

\begin{lemma} \label{invariant}
	$\varOmega$ is a $Mod(S)$--invariant subset of $\ELS$ and $F'$ is $Mod(S)$--equivariant.
\end{lemma}
\begin{proof} To prove the first statement, let $L \in \varOmega$, $f \in Mod(S)$, and  $F'(L)=[\{c_n\}]$. We need to show that $f(L) \in \varOmega$. Any Hausdorff accumulation points of $\{f(c_n)\}$ are precisely the $f$--image of the Hausdorff accumulation points of $\{c_n\}$, and hence all contain $f(L)$. Since $F':\varOmega \to \partial\CS$ is a bijection, $[\{fc_n\}] = F'(L_0)$, for some $L_0 \in \varOmega \subseteq \ELS$ ($L_0$ is the unique ending lamination such that any Hausdorff accumulation point of $\{fc_n\}$ contains $L_0$). Since $f(L)$ is an ending lamination, we must have $f(L)=L_0 \in \varOmega$, as required. 

Since  $f(F'(L)) = f([\{c_n\}]) = [\{f(c_n) \}] =F'(L_0) =F'(f(L))$, $F'$ is $Mod(S)$--equivariant. 
\end{proof}
 
 The proof of the following lemma is essentially the same as Theorem 6.19 of \cite{FLP}, so we omit it here.
 
\begin{lemma} \label{dense}
	For $L \in\ELS$, $\overline{Mod(S)\cdot L}=\ELS$.  
\end{lemma}
\begin{proof}[Proof of Theorem \ref{main1} (closed case)]
Since $F$ is a homeomorphism, by Lemma \ref{closed2}, $\varOmega\subseteq\mathcal{EL}(S\smallsetminus x)$ is a closed subset, and so is closed in $\ELS$. By Lemma \ref{invariant} and \ref{dense}, $\varOmega = \ELS$. Thus, $F':\ELS\to \partial\CS$ is a homeomorphism which is $Mod(S)$--equivariant by Lemma \ref{invariant}.
\end{proof}
\begin{remark*}
	{\em It seems likely that one could also gives a direct proof in the closed case, using {\em bicorn paths} introduced in \cite{PA}. To avoid developing this theory, we gave this alternative proof.}
\end{remark*}

\bibliographystyle{amsplain}
\bibliography{arcref}

\end{document}